\documentclass[12pt]{amsart}

\usepackage{amsmath}
\usepackage[mathscr]{eucal}
\usepackage{amssymb}
\usepackage{latexsym}
\usepackage{amsthm}

\renewcommand\L{\mathcal{L}}
\renewcommand\P{\mathcal{P}}

\newcommand\N{\mathbb{N}}
\newcommand\Q{\mathbb{Q}}

\newcommand\Z{\mathbb{Z}}

\newtheorem{theorem}{Theorem}[section] 
\newtheorem{proposition}[theorem]{Proposition} 
\newtheorem{lemma}[theorem]{Lemma} 
\newtheorem{corollary}[theorem]{Corollary} 

\newtheorem{thmA}{Theorem}

\newtheorem{thmB}{Theorem}

\theoremstyle{definition} 
  
\newtheorem{definition}[theorem]{Definition} 
 
\newtheorem{example}[theorem]{Example}

\theoremstyle{plain}

\def\ds{\displaystyle}
\def\p{\varphi}

\begin{document}

\title[Weighted composition operators between spaces on tree]
{Weighted composition operators acting from the Lipschitz space to the space of bounded functions on a tree}

%----------Author 1
\author[T. Hosokawa]{Takuya Hosokawa}
\address{College of Engineering, Ibaraki University, 4-12-1, Nakanarusawa, Hitachi, Ibaraki, 316-8511, Japan}
\email{takuya.hosokawa.com@vc.ibaraki.ac.jp}

\thanks{The author is partially supported by Grant-in-Aid for Scientific Research, Japan Society for the Promotion of Science (No.16K05190).}

\begin{abstract}  
We study the weighted composition operators between the Lipschitz space and the space of bounded functions on the set of vertices of an infinite tree. 
We characterized the boundedness, the compactness, and the boundedness from below of weighted composition operators. 
We also determine the isometric weighted composition operators. 
\end{abstract}


\maketitle

%%%%%%%%%%%%%%%%%%%%%%%%%%%%%%%%%%%%%%%%%%%%%%%%%%%%%%%%%%%%
\section{Introduction}

A graph $G= (V, E)$ is a pair of a nonempty set $V$, which is called the vertex set, and a subset $E$ of $\{ \, [v, w] \in V \times V : v \not= w \}$, which is called the edge set. 
Two vertices $v$ and $w$ are called neighbors if $[v, w] \in E$, and we denote by $v \sim w$. 
For any vertex $v \in V$, let $\deg v$ denote the number of neighbors of $v$. 
The vertex $v$ is called terminal if $\deg v =1$. 
A path is a finite or infinite sequence $[v_0 , v_1 , \dots ]$ of vertices such that $v_k \sim v_{k+1}$ and $v_k \not= v_{k+2}$ for all $k$. 

A tree is a connected, locally finite, undirected graph with no loop, no cycle. 
Remark that for any two vertices $v, w$ there exists a unique finite path from $v$ to $w$. 
For any vertex $v , w \in V$, let the distance between $v$ and $w$ be the numbers of edges in the finite path from $v$ to $w$, and we denote by $d_T (v, w)$. 
In this paper, we assume that the tree $T$ is rooted at a vertex $o$ and has no terminal. 
Hence $T$ is an infinite graph. 

A vertex $v$ is called descendant of $w$ if $w$ lies in the unique path from $o$ to $v$. 
Then $w$ is called an ancestor of $v$. 
The set $S_v$ consisting of $v$ and all descendants of $v$ is called the sector determined by $v$. 
For $v \in T$, define that $|v| = d_T (o, v)$. 
We use the notation $T^* = T \setminus \{ o \}$. 
The parent $v^-$ of $v \in T^*$ is the unique vertex satisfying $v^- \sim v$ and $|v^-| = |v|-1$. 
Remark that for any neighbors $v \sim w$, we have the alternative of $v= w^-$ or $w = v^-$. 

In \cite{CC}, Cohen and Colonna pointed out the geometric correspondence between the hyperbolic disk and the homogeneous trees, that is, each vertex has the same number of neighbors. 
The authors described the relation in terms of the M\"{o}bius transformations and the hyperbolic tilings. 
In this line, we can regard the complex-valued functions on the vertices of trees (possibly non-homogeneous) as a discretization of the functions on the unit disk. 

The supremum norm of $f$ on $T$ will be denoted by 
\begin{equation*}
\| f \|_\infty = \sup_{v \in T} |f(v)| , 
\end{equation*} 
and the space of the bounded functions on $T$ by $L^\infty$. 
The discrete derivative of $f$ is defined by 
\begin{equation*}
D f (v) = 
\left\{ 
\begin{matrix} 
f(v) - f(v^-) & (v \not= o), \\ 
\ 0 & (v=o). 
\end{matrix}
\right.
\end{equation*} 
The set of all functions $f$ on T such that $\| Df \|_\infty < \infty$ is called the Lipschitz space, denote by $\L$. 
Then $\L$ is a Banach space under the norm 
\begin{equation*}
\| f \|_\L = |f(o)| + \| D f \|_\infty . 
\end{equation*} 
Since $\| Df \|_\infty \leq 2 \| f \|_\infty$, we have that $L^\infty \subset \L$. 
It is known that the Lipschitz functions $f$ follows the growth condition: 
\begin{equation} \label{growth}
|f(v)| \leq |f(o)| + |v| \cdot \| Df \|_\infty 
\end{equation}
(see also Lemma 3.4 of \cite{CE1}). 

Let the little Lipschitz space $\L_o$ be the subspace of $\L$ consisting of all functions $f$ with $\lim_{|v| \to \infty} Df(v) = 0$. 
Colonna and Easley \cite{CE1} proved that $\L_o$ is the closure in $\L$ of the set 
\begin{equation*}
\P = \left\{ \sum_{k=1}^n a_k \eta_{v_k} : n \in \N , v_k \in T , a_k \in \Q [i] \right\} , 
\end{equation*} 
where $\eta_v$ is the characteristic function on the sector $S_v$ determined by $v$. 
Thus $\L_o$ is a closed separable subspace of $\L$.

Let $\psi$ be a function on $T$. 
The multiplication operator $M_\psi$ is defined by 
\begin{equation*}
M_\psi f (v) = \psi (v) f(v) . 
\end{equation*}
For a self-map $\p$ of $T$, the composition operator $C_\p$ is defined by 
\begin{equation*}
C_\p f (v) = f( \p (v)) . 
\end{equation*}
Moreover, we define the weighted composition operator $\psi C_\p$ by 
\begin{equation*}
\psi C_\p f (v) = M_\psi C_\p f (v) = \psi (v) f( \p (v)) . 
\end{equation*}

For the multiplication operators, the boundedness and the compactness has been studied on some function spaces. 
Allen and Craig \cite{AC} characterized the boundedness and the compactness of multiplication operators on the weighted Banach space $L_\mu^\infty$ on $T$. 
We here present the results for the case of $L^\infty$. 
\begin{thmA}[\cite{AC}] 
Let $\psi$ be a function on $T$. 
\begin{enumerate}
\item $M_\psi$ is bounded on $L^\infty$ if and only if $\psi \in L^\infty$. 
Moreover, if $M_\psi$ is bounded on $L^\infty$, then $\| M_\psi \|_{L^\infty} = \| \psi \|_\infty$. 
\item $M_\psi$ is compact on $L^\infty$ if and only if $\psi (v) \to 0$ as $|v| \to \infty$. 
\end{enumerate}
\end{thmA}
Colonna and Easley \cite{CE2} characterized the boundedness and the compactness of the multiplication operators acting from the Lipschitz spaces to $L^\infty$. 
\begin{thmB}[\cite{CE2}] 
Let $\psi$ be a function on $T$. 
\begin{enumerate}
\item The following conditions are equivalent. 
\begin{enumerate} 
\item $M_\psi : \L \to L^\infty$ is bounded. 
\item $M_\psi : \L_o \to L^\infty$ is bounded. 
\item $\ds \sup_{v \in T} |v| | \psi (v)| < \infty$. 
\end{enumerate}
\item Suppose that $M_\psi : \L \to L^\infty$ is bounded. 
Then the following conditions are equivalent. 
\begin{enumerate} 
\item $M_\psi : \L \to L^\infty$ is compact. 
\item $M_\psi : \L_o \to L^\infty$ is compact. 
\item $\ds \lim_{|v| \to \infty} |v| |\psi (v)| \to 0$. 
\end{enumerate}
\end{enumerate}
\end{thmB}
In \cite{ACE}, Allen, Colonna, and Easley also studied the composition operators on the Lipschitz space $\L$. 

The following is called the weak convergence lemma. 
Allen and Craig introduced similar result on more general settings (Lemma 2.5 in \cite{AC}). 
\begin{lemma} \label{WCL}
Let $X$ and $Y$ be $\L$, $\L_o$, or $L^\infty$. Let $A$ be a bounded operator acting from $X$ to $Y$. 
Then $A$ is compact if and only if $\| A f_n \|_Y \to 0$ for every bounded sequence $\{ f_n \}$ in $X$ converging to $0$ pointwise. 
\end{lemma}

Recall that a bounded linear operator $T : X \to Y$, where $X$ and $Y$ are two Banach spaces, is bounded below if there exists a positive constant $M$ such that 
\begin{equation*} 
\| T x \|_Y \geq M \| x \|_X 
\end{equation*}
for any $x \in X$. 
In \cite{Mul}, M\"{u}ller introduced two kinds of minimum moduli of bounded linear operators, one of which is related to the boundedness from below. 
\begin{definition} (p.86, \cite{Mul}) 
Let $X, Y$ be two Banach spaces and $T$ be a bounded linear operator from $X$ to $Y$. 
\begin{enumerate}
\item The injectivity modulus $j(T)$ of $T$ is defined by 
\begin{equation*}
j(T) = \inf \{ \| T x \|_Y : \| x \|_X = 1 \}. 
\end{equation*}
\item The surjectivity modulus $k(T)$ of $T$ is defined by 
\begin{equation*}
k(T) = \sup \{ r \geq 0 : T(U_X) \supset r \, U_Y \}. 
\end{equation*}
\end{enumerate}
\end{definition}
It is known that $j(T)>0$, that is, $T$ is bounded below, if and only if $T$ is injective and the range of $T$ is closed in $Y$. 
It is also known that $k(T)>0$ if and only if $T$ is surjective. 
See \cite{Mul} for more properties of those minimum moduli in general setting, and \cite{Hos} for the estimate on the minimum moduli of weighted composition operators on $H^\infty$.

Our main purpose is to study the weighted composition operators $\psi C_\p$ acting between $\L$ and $L^\infty$. 
In section 2, we will study the boundedness, the compactness of $\psi C_\p$ on $L^\infty$. 
We also characterize the isometric weighted composition operators, and the boundedness below of $\psi C_\p$ on $L^\infty$. 
Moreover, we determine the operator norm, the essential norm, and the minimum moduli of $\psi C_\p$ on $L^\infty$. 
In section 3, we will study the boundedness, the compactness of $\psi C_\p$ acting from $\L$ and $\L_o$ to $L^\infty$. 
We also show that there is no isometric weighted composition operator from $\L$ and $\L_o$ to $L^\infty$. 
Those results are generalizations of the results on $M_\psi$ given in \cite{CE2}. 
Moreover, we will give the estimate on the minimum moduli of $\psi C_\p$ from $\L$ to $L^\infty$.

%%%%%%%%%%%%%%%%%%%%%%%%%%%%%%%%%%%%%%%%%%%%%%%%%%%%%%%%%%%%%%%%%%%%%%%%%%%%%%%%%%
\section{Results on $\psi C_\p : L^\infty \to L^\infty$} 

%%%%%%%%%%%%%%%%%%%%%%%%%%%%%%%%%%%%%%%%%
\subsection{Boundedness and compactness} 
\ 

%First we consider the weighted composition operators $\psi C_\p$ on $L^\infty$. 
\begin{theorem} \label{bdd2-1}
Let $\psi$ be a function on $T$, and $\p$ be a self-map of $T$. 
\begin{enumerate}
\item $\psi C_\p$ is bounded on $L^\infty$ if and only if $\psi \in L^\infty$. 
Moreover, if $\psi C_\p$ is bounded on $L^\infty$, then $\| \psi C_\p \| = \| \psi \|_\infty$. 
\item Suppose that $\psi C_\p$ is bounded on $L^\infty$. 
Then $\psi C_\p$ is compact on $L^\infty$ if and only if one of the following conditions holds: 
\begin{enumerate} 
\item $\p (T)$ is a finite subset of $T$, 
\item $\psi (v) \to 0$ if $| \p (v)| \to \infty$. 
\end{enumerate}
Moreover, if $\p (T)$ is an infinite subset of $T$, the following estimate holds: 
\begin{equation} \label{ess-norm-Linfty} 
% \limsup_{|\p (v)| \to \infty} |p(v)| \leq 
\| \psi C_\p \|_e 
= \limsup_{|\p (v)| \to \infty} | \psi (v)| . 
\end{equation}
\end{enumerate}
\end{theorem}

\begin{proof}
(i) If $\psi C_\p$ is bounded on $L^\infty$, then $\psi = \psi C_\p \, 1 \in L^\infty$. 
On the other hand, if $\psi \in L^\infty$, then we have for $f \in L^\infty$ with $\| f \|_\infty =1$, 
\begin{equation*} 
\| \psi C_\p f \|_\infty = \| \psi (v) f( \p (v)) \|_\infty \leq \| \psi (v) \|_\infty . 
\end{equation*}
Hence we have $\psi C_\p$ is bounded on $L^\infty$ if and only if $\psi \in L^\infty$. 

Next we let $\chi_{\p (w)}$ be the characteristic function at $\p (w) \in T$. 
Then we get 
\begin{equation*}
\| \psi C_\p \| \geq \| \psi (v) \chi_{\p (w)} (\p (v)) \|_\infty \geq | \psi (w) \chi_{\p (w)} (\p (w))| = | \psi (w)| . 
\end{equation*} 
Taking the supremum over all $w \in T$, we obtain 
\begin{equation*}
\| \psi C_\p \| \geq \| \psi \|_\infty . 
\end{equation*}
We conclude that $ \| \psi C_\p \| = \| \psi \|_\infty$.

(ii) We remark that $\| \psi \|_\infty < \infty$ since $\psi C_\p$ is bounded on $L^\infty$. 
Let $\{ f_n \}$ be an arbitrary bounded sequence in $L^\infty$ converging to $0$ pointwise. 
Write $\| f_n \|_\infty < M$ for any $n$. 

Suppose that $\p (T)$ is a finite set. 
We have 
\begin{equation*}
\| \psi C_\p f_n \|_\infty 
= \sup_{v \in T} | \psi (v) f_n (\p (v)) | 
\leq \| \psi \|_\infty \cdot \sup_{v \in \p (T)} |f_n (v)| \to 0 . 
\end{equation*} 
By Lemma \ref{WCL}, we have $\psi C_\p$ is compact on $L^\infty$. 

We here suppose that $\p (T)$ is an infinite subset of $T$. 
Then it is enough to prove \eqref{ess-norm-Linfty}. 
Let $K$ be a compact operator on $L^\infty$. 
For any sequence $\{ v_n \}$ in $T$ such that $|\p (v_n)| \to \infty$, we have 
\begin{equation*}
\| ( \psi C_\p +K) \| 
\geq \| ( \psi C_\p +K) \chi_{\p(v_n)} \|_\infty 
\geq | \psi (v_n)| - \| K \chi_{\p (v_n)} \|_\infty . 
\end{equation*} 
Taking the limit superior over $n$ and the infimum over all compact operators $K$, we obtain 
\begin{equation*}
\| \psi C_\p \|_e \geq \limsup_{n \to \infty} | \psi (v_n)| . 
\end{equation*} 

Next we prove the converse. 
Let $N$ be a positive integer and define the operator $K_N$ by 
\begin{equation*}
K_N f (v) = \left\{ 
\begin{matrix} 
f(v) & \ (|v| \leq N ) \\ 
0 & \ ( |v| >N ) 
\end{matrix} \right. . 
\end{equation*}
It is easy to see that $K_N$ is compact on $L^\infty$ (see the proof of Theorem 3.7 in \cite{CE2}). 
Since $\psi C_\p K_N$ is also compact on $L^\infty$, we have 
\begin{eqnarray*}
\| \psi C_\p \|_e 
& \leq & \| ( \psi C_\p - \psi C_\p K_N) \| 
\ \leq \ \sup_{\| f \|_\infty =1} \| ( \psi C_\p - \psi C_\p K_N) f \|_\infty \\ 
& = & \sup_{\| f \|_\infty =1} \sup_{| \p (v) | > N} | \psi (v) f( \p (v)) | 
\ \leq \ \sup_{| \p (v) | > N} | \psi (v)| . 
\end{eqnarray*} 
Letting $N \to \infty$, we get 
\begin{equation*}
\| \psi C_\p \|_e \leq \limsup_{n \to \infty} | \psi (v_n)| . 
\end{equation*} 
This completes the proof. 
\end{proof}

\begin{corollary} \label{cor-norm-Linfty}
Let $\psi$ be a function on $T$ and $\p$ be a self-map of $T$. 
\begin{enumerate}
\item $\| M_\psi \| = \| \psi \|_\infty$. 
\item $\| C_\p \| = 1$. 
\end{enumerate}
\end{corollary} 

Recall that a tree $T$ is locally finite, and hence the self-map $\p$ of $T$ has finite range if and only if 
\begin{equation*} 
\sup_{v \in T} |\p (v)| < \infty . 
\end{equation*} 

\begin{corollary} \label{cor-compact-Linfty}
Let $\psi$ be a bounded function on $T$ and $\p$ be a self-map of $T$. 
\begin{enumerate}
\item $M_\psi$ is compact on $L^\infty$ if and only if 
\begin{equation*}
\limsup_{|v| \to \infty} \, | \psi (v)| =0 . 
\end{equation*}
\item $C_\p$ is compact on $L^\infty$ if and only if $\p$ has finite range in $T$. 
\end{enumerate}
\end{corollary} 

Moreover, we have the zero-one law on the essential norm of $\psi C_\p$. 
\begin{corollary} 
Let $\p$ be a self-map of $T$. 
Then 
\begin{equation*}
\| C_\p \|_e = \left\{ 
\begin{matrix} 
0 & \ ( \mbox{if $\p$ has finite range} ) \hspace{18mm} \\ 
1 & \ ( \mbox{if $\p$ does not have finite range} ) 
\end{matrix} \right. . 
\end{equation*}
\end{corollary}

%%%%%%%%%%%%%%%%%%%%%%%%%%%%%%%%%%%%%%%%%%%%%%%%%%%%
\subsection{Isometries}

\begin{theorem} \label{isometry1}
Let $\psi$ be a function on $T$ and $\p$ be a self-map of $T$. 
Then $\psi C_\p$ is an isometry on $L^\infty$ if and only if $\p$ is surjective and 
\begin{equation} \label{equ-iso1}
\sup_{v \in \p^{-1} (w)} | \psi (v)| =1 
\end{equation}
for any $w \in T$. 
\end{theorem}

\begin{proof}
Suppose $\psi C_\p$ is an isometry on $L^\infty$. 
Since $\| \chi_w \|_\infty =1$ for $w \in T$, we have 
\begin{eqnarray*}
1 & = & \| \psi C_\p \chi_w \|_\infty 
\ = \ \sup_{v \in T} | \psi (v) \cdot \chi_w ( \p (v))| \\ 
& = & \left\{ 
\begin{matrix} 
0 & \ ( \mbox{if $w \not\in \p (T)$} ) \\ 
\displaystyle \sup_{v \in \p^{-1} (w)} | \psi (v)| & \ ( \mbox{if $w \in \p (T)$} )
\end{matrix} \right. . 
\end{eqnarray*}
Thus we get $\p$ is surjective and \eqref{equ-iso1} for any $w \in T$. 

To prove the converse, we assume $\p$ is surjective and \eqref{equ-iso1} for any $w \in T$. 
It follows that $\| \psi \|_\infty =1$, and hence we have $\| \psi C_\p f \|_\infty \leq \| f \|_\infty$ for any $f \in L^\infty$. 
Let $\{ w_n \}$ be a sequence in $T$ such that $|f(w_n)| \to \| f \|_\infty$. 
Then we have 
\begin{eqnarray*}
\| \psi C_\p f \|_\infty 
& \geq & \lim_{n \to \infty} \sup_{v \in \p^{-1} (w_n)} | \psi (v) \cdot f( \p (v))| \\ 
& = & \lim_{n \to \infty} \left( \sup_{v \in \p^{-1} (w_n)} | \psi (v)| \right) \cdot |f(w_n)| \\ 
& = & \lim_{n \to \infty} |f(w_n)| = \| f \|_\infty . 
\end{eqnarray*}
We get $\psi C_\p$ is an isometry on $L^\infty$. 
\end{proof}

\begin{example}
Let $\Z$ be the set of all integers. 
Define $\p$ be a self-map of $\Z$ by 
\begin{equation*}
\p (n) = \left\{ 
\begin{matrix} 
n & \ ( n \geq 0 ) \hspace{18mm} \\ 
-n & \ ( n = -1, -3, \cdots ) \vspace{2mm} \\ 
\ds \frac{\, n \,}{2} & \ ( n = -2, -4, \cdots) 
\end{matrix} \right. . 
\end{equation*}
Clearly, $\p$ maps $\Z$ onto $\Z$. 
Moreover, we define a function $\psi$ on $\Z$ by 
\begin{equation*}
\psi (n) = \left\{ \, 
\begin{matrix} 
0 & \ ( n = -1, -3, \cdots ) \vspace{2mm} \\ 
1 & \ ( \mbox{otherwise} ) \hspace{10mm}
\end{matrix} \right. . 
\end{equation*} 
Since $\psi$ holds \eqref{equ-iso1} for any $w \in \Z$, we have $\psi C_\p$ is an isometry on $L^\infty (\Z)$. 
\end{example}

\begin{corollary} \label{cor-iso-Linfty}
Let $\psi$ be a function on $T$ and $\p$ be a self-map of $T$. 
\begin{enumerate}
\item $M_\psi$ is an isometry on $L^\infty$ if and only if $| \psi | \equiv 1$ on $T$. 
\item $C_\p$ is an isometry on $L^\infty$ if and only if $\p$ is surjective. 
\end{enumerate}
\end{corollary}

%%%%%%%%%%%%%%%%%%%%%%%%%%%%%%%%%%%%%%%%%%%%%%%%
\subsection{Boundedness from below and minimum moduli} \ 

We determine the minimum moduli $j( \psi C_\p )$ and $k( \psi C_\p )$ for bounded weighted composition operators $\psi C_\p$ on $L^\infty$. 
\begin{theorem} \label{inj-mod}
Let $\psi$ be a bounded function on $T$ and $\p$ be a self-map of $T$. 
\begin{enumerate}
\item If $\p$ is not surjective, then $j( \psi C_\p) =0$. 
\item If $\p$ is surjective, then 
\begin{equation} \label{est-inj}
j( \psi C_\p) = \inf_{w \in T} \left( \sup_{v \in \p^{-1} (w)} | \psi (v)| \right) . 
\end{equation}
\end{enumerate}
\end{theorem}

\begin{proof} 
(i) Suppose $\p$ is not surjective. 
For $w \not\in \p (T)$, we have $\psi C_\p \chi_w \equiv 0$ on $T$. 
This implies that $j( \psi C_\p) =0$. 

(ii) Suppose $\p$ is surjective. 
Let $M$ be the infimum of right side of \eqref{est-inj}. 
For any $w \in T$, the preimage $\p^{-1} (w)$ is not the empty set and 
\begin{equation*}
\| \psi C_\p \chi_w \|_\infty = \sup_{v \in \p^{-1} (w)} | \psi (v)| . 
\end{equation*}
Thus we get 
\begin{equation*}
j( \psi C_\p) \leq \inf_{w \in T} \, \| \psi C_\p \chi_w \|_\infty = M . 
\end{equation*}

Next we will show $j( \psi C_\p ) \geq M$. 
For any $w \in T$ and any $f \in L^\infty$ with $\| f \|_\infty =1$, we have 
\begin{eqnarray*}
\| \psi C_\p f \|_\infty 
& = & \sup_{v \in T} | \psi (v) \cdot f ( \p (v))| \\ 
& \geq & \sup_{v \in \p^{-1} (w)} | \psi (v) \cdot f ( \p (v))| \\ 
& = & \left( \sup_{v \in \p^{-1} (w)} | \psi (v)| \right) \cdot | f (w)| 
\ \geq \ M \cdot | f (w)| . 
\end{eqnarray*}
Since $w$ is arbitrary, we have $ \| \psi C_\p f \|_\infty \geq M \| f \|_\infty$. 
Then we get $j( \psi C_\p) \geq M$. 
\end{proof}

\begin{corollary} \label{cor-bdd-below-Linfty}
Let $\psi$ be a bounded function on $T$ and $\p$ be a self-map of $T$. 
Then $\psi C_\p$ is bounded below on $L^\infty$ if and only if $\p$ is surjective on $T$ and 
\begin{equation*} 
\inf_{w \in T} \left( \sup_{v \in \p^{-1} (w)} | \psi (v)| \right) > 0 . 
\end{equation*}
\end{corollary} 

\ 

We here consider the surjectivity modulus of $\psi C_\p$ on $L^\infty$. 
First, we will show that if $\psi$ has zeros on $T$, then $\psi C_\p$ is not surjective on $L^\infty$. 
\begin{proposition} \label{prop-sur}
Let $\psi$ be a bounded function on $T$ and $\p$ be a self-map of $T$. 
If there exists $v_0 \in T$ such that $\psi ( v_0 ) =0$, then $k( \psi C_\p ) = 0$. 
\end{proposition}

\begin{proof} 
Suppose that $k( \psi C_\p ) >0$. 
For $0<r < k( \psi C_\p )$, there exists $f \in L^\infty$ such that $\| f \|_\infty \leq 1$ and $\psi C_\p f = r \cdot \chi_{v_0}$. 
Then we have 
\begin{equation*}
r = r \cdot \chi_{v_0} (v_0) = \psi (v_0) \cdot f (\p (v_0)) =0, 
\end{equation*} 
which is a contradiction. 
\end{proof}

We determine the surjectivity modulus of $\psi C_\p$. 
\begin{theorem} \label{sur-mod}
Let $\psi$ be a bounded function on $T$ and $\p$ be a self-map of $T$. 
\begin{enumerate}
\item If $\p$ is not injective, then $k( \psi C_\p) =0$. 
\item If $\p$ is injective, then 
\begin{equation*} 
k( \psi C_\p) = \inf_{v \in T} | \psi (v)| . 
\end{equation*}
\end{enumerate}
\end{theorem}

\begin{proof} 
(i) Let $\p$ be a self-map of $T$, which is not injective. 
There exist two distinct vertices $v_1 , v_2 \in T$ such that $\p (v_1) = \p (v_2)$. 
Suppose that $k( \psi C_\p) >0$ and let $0<r< k( \psi C_\p)$. 
Then there exists $f \in L^\infty$ such that $\| f \|_\infty \leq 1$ and $\psi C_\p f = r \cdot \chi_{v_1}$. 
Since 
\begin{equation*}
r = r \cdot \chi_{v_1} (v_1) = \psi (v_1) \cdot f_1 (\p (v_1)) , 
\end{equation*} 
we have $\psi (v_1) \not= 0$ and $f_1 ( \p (v_1)) \not= 0$. 
On the other hand, we have 
\begin{equation*} \label{est-sur}
0 = r \cdot \chi_{v_1} (v_2) = p(v_2) \cdot f_1 (\p (v_2)) = p(v_2) \cdot f_1 (\p (v_1)) . 
\end{equation*}
Thus we get $p (v_2) =0$. 
By Proposition \ref{prop-sur}, we get $k( \psi C_\p) = 0$. 
This is a contradiction. 

(ii) Suppose $\p$ is injective. 
To prove $\displaystyle k( \psi C_\p) \leq \inf_{v \in T} | \psi (v)|$, we may assume $k( \psi C_\p) > 0$. 
Let $\varepsilon$ be a positive number less than $k( \psi C_\p)$. 
For any $w \in T$, there exists $f \in L^\infty$ such that $\| f \|_\infty \leq 1$ and $\psi C_\p f = (k( \psi C_\p) - \varepsilon) \cdot \chi_w$. 
It follows that 
\begin{equation*} 
k( \psi C_\p) - \varepsilon = (k( \psi C_\p) - \varepsilon) \cdot \chi_w (w) = \psi (w) \cdot f( \p (w)) \leq | \psi (w)| . 
\end{equation*} 
Since $w$ is arbitrary, we get $k( \psi C_\p) - \varepsilon \leq \inf_{v \in T} | \psi (v)|$. 
Letting $\varepsilon \to 0$, we get $ k( \psi C_\p) \leq \inf_{v \in T} | \psi (v)|$. 

Next we will prove the converse. 
We may assume $M = \inf_{v \in T} | \psi (v)| >0$. 
It is enough to prove that, for any $g \in L^\infty$ with $\| g \|_\infty \leq M$, there exists a function $f \in L^\infty$ such that $\| f \|_\infty \leq 1$ and $\psi C_\p f = g$. 
Remark that $\p$ is injective and $\psi (v) \not= 0$ for any $v \in T$. 
We here define that 
\begin{equation*}
f(w) = \left\{ 
\begin{matrix} 
\displaystyle \frac{g( \p^{-1} (w))}{\psi ( \p^{-1} (w))} & \ ( w \in \p(T) ) \vspace{2mm} \\ 
0 & \ ( w \not\in \p(T) ) 
\end{matrix} \right. 
\end{equation*}
Then we can see $\psi C_\p f (v) = g(v)$ and 
\begin{eqnarray*}
\| f \|_\infty 
& = & \sup_{w \in \p (T)} \frac{|g( \p^{-1} (w))|}{|\psi ( \p^{-1} (w))|} 
\ \leq \ \frac{\| g \|_\infty}{\displaystyle \inf_{v \in T} | \psi (v)|} \\ 
& \leq & \frac{M}{M} 
\ = \ 1 
\end{eqnarray*}
Thus we conclude $ k( \psi C_\p) = \inf_{v \in T} | \psi (v)|$. 
\end{proof}

\begin{corollary} \label{cor-minimum-Linfty}
Let $\psi$ be a bounded function on $T$ and $\p$ be a self-map of $T$. 
\begin{enumerate}
\item $\displaystyle j(M_\psi) = k(M_\psi) = \inf_{v \in T} | \psi (v)|$. \vspace{2mm} 
\item $j(C_\p) = \left\{ 
\begin{matrix} 
0 & \ ( \mbox{if $\p$ is not surjective} ) \\ 
1 & \ ( \mbox{if $\p$ is surjective} ) \hspace{7mm} 
\end{matrix} \right.$. \vspace{2mm} 
\item $k(C_\p) = \left\{ 
\begin{matrix} 
0 & \ ( \mbox{if $\p$ is not injective} ) \\ 
1 & \ ( \mbox{if $\p$ is injective} ) \hspace{7mm} 
\end{matrix} \right.$. 
\end{enumerate}
\end{corollary}

%%%%%%%%%%%%%%%%%%%%%%%%%%%%%%%%%%%%%%%%%%%%%%%%%%%%%%%%%%%%%%%%%%%%%%%%%%%%%%
\section{Results on $\psi C_\p : \L \to L^\infty$}

In this section we study the weighted composition operator $\psi C_\p$ acting from $\L$ to $L^\infty$. 

%%%%%%%%%%%%%%%%%%%%%%%%%%%%%%%%%%%%%%%%%
\subsection{Boundedness and compactness} 
\ 

\begin{theorem} \label{bdd1}
Let $\psi$ be a function on $T$ and $\p$ be a self-map of $T$. 
Then the following are equivalent: 
\begin{enumerate}
\item $\psi C_\p : \L \to L^\infty$ is bounded. 
\item $\psi C_\p : \L_o \to L^\infty$ is bounded. 
\item $\psi \in L^\infty$ and $\ds \sup_{v \in T} | \psi (v)| |\p (v)|< \infty$. 
\end{enumerate}
Furthermore, under the above conditions, the following estimate holds: 
\begin{equation*}
\max \{ \| \psi \|_\infty , \| \psi (v) \cdot |\p (v)| \|_\infty \} 
\leq \| \psi C_\p \|_{\L \to L^\infty} 
\leq \| \psi (v) (1+ | \p (v)|) \|_\infty . 
\end{equation*}
\end{theorem}

\begin{proof}
The implication (i) $\Rightarrow$ (ii) is trivial. 
Suppose the condition (iii) holds. 
Let $f$ be a function in $\L$ with $\| f \|_\L \leq 1$. 
Then the growth condition \eqref{growth} implies that 
\begin{eqnarray*}
| \psi C_\p f (v) | 
& = & | \psi (v) f( \p (v))| 
\ \leq \ | \psi (v)| \big( |f(o)| + |\p (v)| \cdot \| Df \|_\infty \big) \\ 
& \leq & | \psi (v)| (1+ | \p (v)|) \| f \|_\L \\ 
& \leq & \big\{ \| \psi \|_\infty + \sup_{v \in T} | \psi (v) \cdot \p (v)| \big\} \| f \|_\L < \infty . 
\end{eqnarray*}
Hence we get (i) and the upper bound of $\| \psi C_\p \|_{\L \to L^\infty}$. 

Next, suppose the condition (ii) holds. 
Then we get $\psi C_\p 1 = \psi \in L^\infty$. 
For any positive integer $N$, we let 
\begin{equation} \label{FN}
F_N (v) = \left\{ 
\begin{matrix} 
|v| & ( |v| \leq N) \\ 
N & ( |v| > N) 
\end{matrix} \right. . 
\end{equation}
Clearly $F_N \in \L_o$ and $\| F_N \|_\L =1$. 
For any vertex $w \in T$, take enough large $N$ so that $N > |\p (w)|$. 
Then we have that 
\begin{equation*}
\| \psi C_\p \|_{\L_o \to L^\infty} \ \geq \ \| \psi C_\p F_N \|_\infty 
\ \geq \ | \psi (w) F_N (\p (w)) | 
\ = \ | \psi (w)| |\p (w)| . 
\end{equation*}
Taking the supremum on $w$ over $T$, we conclude the condition (iii) and the lower bound of $\| \psi C_\p \|_{\L \to L^\infty}$. 
\end{proof}

\begin{theorem} \label{compact1}
Let $\psi$ be a function on $T$ and $\p$ be a self-map of $T$. 
Suppose that $\psi C_\p : \L \to L^\infty$ is bounded. 
Then the following are equivalent: 
\begin{enumerate}
\item $\psi C_\p : \L \to L^\infty$ is compact. 
\item $\psi C_\p : \L_o \to L^\infty$ is compact. 
\item $\displaystyle | \psi (v)| \cdot |\p (v)| \to 0$ if $|\p (v)| \to \infty$. 
\end{enumerate}
\end{theorem}
Since the above theorem follows from Theorem \ref{ess-norm-wc}, here we do not give its proof. 
We here present the example satisfying the conditions of Theorem \ref{bdd1} and Theorem \ref{compact1}. 
\begin{example}
Let $\p$ be a self-map of $T$ with infinite range. 
Then $C_\p$ is not bounded from $\L$ to $L^\infty$. 
Put 
\begin{equation*}
\psi (v) = \frac{1}{|\p (v)|+1} 
\end{equation*}
Then, by Theorem \ref{bdd1} and Theorem \ref{compact1}, we have that $\psi C_\p$ is bounded but is not compact. 
On the other hand, then $\psi^2 C_\p$ is compact. 
\end{example}

\begin{theorem} \label{ess-norm-wc}
Let $\psi$ be a function on $T$ and $\p$ be a self-map of $T$. 
Suppose that $\psi C_\p$ is a bounded operator from $\L$ (respectively, $\L_o$) to $L^\infty$. 
Then the following holds: 
\begin{equation*}
\| \psi C_\p \|_{e, \L \to L^\infty} 
= \| \psi C_\p \|_{e, \L_o \to L^\infty} 
= \limsup_{|\p (v)| \to \infty} \, | \psi (v)| \cdot |\p (v)| . 
\end{equation*}
\end{theorem}

\begin{proof}
It is trivial that $\| \psi C_\p \|_{e, \L \to L^\infty} \geq \| \psi C_\p \|_{e, \L_o \to L^\infty}$. 

For each positive integer $n$, we define the operator $K_n : \L \to \L$ by 
\begin{equation*}
K_n f (v) = \left\{ 
\begin{matrix} 
f(v) & \ (|v| \leq n ) \\ 
f (v_n) & \ ( |v| >n ) 
\end{matrix} \right. . 
\end{equation*}
where $v_n$ is the unique vertex lying in the path between $o$ and $v$ such that $|v_n| =n$. 
By Lemma \ref{WCL}, we have that $K_n$ is compact. 
Therefore, we get 
\begin{eqnarray*}
\| \psi C_\p \|_{e, \L \to L^\infty} 
& \leq & \| \psi C_\p - \psi C_\p K_n \|_{\L \to L^\infty} \\ 
& = & \sup_{\| f \|_\L =1} \| ( \psi C_\p - \psi C_\p K_n) f \|_\infty 
\end{eqnarray*}
For $| \p (v)|>n$, put $w_n$ the unique vertex lying in the path between $o$ and $\p (v)$ such that $|w_n| =n$. 
Since $d_T ( \p (v) , w_n) = |\p (v)| - n$, we have that 
\begin{eqnarray*}
\| (\psi C_\p - \psi C_\p K_n) f \|_\infty 
& = & \sup_{v \in T}| \, | \psi C_\p f(v) - \psi C_\p K_n f (v) | \\ 
& = & \sup_{| \p (v)| >n} \, | \psi (v)| \cdot | f( \p (v)) - f (w_n) | \\ 
& \leq & \sup_{| \p (v)| >n} \, | \psi (v)| \cdot ( |\p (v)| - n ) \| f \|_\L \\ 
& \leq & \sup_{| \p (v)| >n} \, | \psi (v)| \cdot |\p (v)| \cdot \| f \|_\L . 
\end{eqnarray*}
Thus we have that 
\begin{equation*}
\| \psi C_\p \|_{e, \L \to L^\infty} 
\leq \sup_{| \p (v)| >n} \, | \psi (v)| \cdot |\p (v)| . 
\end{equation*}
Letting $n \to \infty$, we get 
\begin{equation*}
\| \psi C_\p \|_{e, \L \to L^\infty} 
\leq \limsup_{| \p (v)| \to \infty} \, | \psi (v)| \cdot |\p (v)| . 
\end{equation*}

Next we will prove the converse. 
To do this, let $n$ be a positive integer and $r \in (0, 1)$. 
We define that 
\begin{equation*}
g_{n, r} (v) = \left\{ 
\begin{matrix} 
0 & \ (|v| < \sqrt{n} ) \vspace{2mm} \ \ \ \\ 
\frac{n}{n- \sqrt{n}} \cdot \frac{(|v| - \sqrt{n} )^{r+1}}{(n - \sqrt{n} )^r} & \ \ ( \sqrt{n} \leq |v| < n ) \vspace{2mm} \\ 
n & \ ( |v| \geq n ) \ \ \ \ 
\end{matrix} \right. . 
\end{equation*}
Then $g_{n, r}$ is in $\L_o$ and $g_{n, r} (v) \to 0$ pointwise as $n \to \infty$. 
By short calculation, we have 
\begin{equation*}
\| g_{n, r} \|_\L = \sup_{v \in T} |D g_{n, r} (v) | = \frac{n}{n- \sqrt{n}} \cdot \frac{(n - \sqrt{n} )^{r+1} - (n - \sqrt{n} -1)^{r+1}}{(n - \sqrt{n} )^r} . 
\end{equation*}
Since $\{ x^{r+1} - (x-1)^{r+1} \} / x^r \to r+1$ as $x \to \infty$, we have that $\| g_{n, r} \|_\L \to r+1$ as $n \to \infty$. 
By Lemma \ref{WCL}, we have $\| K g_{n, r} \|_\infty \to 0$ as $n \to \infty$ for any compact operator $K$. 
Fix a vertex $w \in T$ and put $n = |\p (w)|$. 
Since $\| \psi C_\p g_{n, r} \|_\infty \geq | \psi (w) \cdot g_{n, r} ( \p (w))| = | \psi (w)| \cdot |\p (w)|$, we obtain that 
\begin{eqnarray*}
\| \psi C_\p - K \|_{\L_o \to L^\infty} 
& \geq & \limsup_{n \to \infty} \frac{\| ( \psi C_\p - K) g_{n, r} \|_\infty}{\| g_{n, r} \|_\L} \\ 
& \geq & \limsup_{n \to \infty} \frac{\| \psi C_\p g_{n, r} \|_\infty - \| K g_{n, r} \|_\infty}{\| g_{n, r} \|_\L} \\ 
& \geq & \frac{1}{r+1} \, \limsup_{|\p (w)| \to \infty} | \psi (w)| \cdot |\p (w)| . 
\end{eqnarray*}
Now, letting $r \to 0$, we get 
\begin{equation*}
\| \psi C_\p \|_{e, \L_o \to L^\infty} 
\geq \limsup_{| \p (v)| \to \infty} \, | \psi (v)| \cdot |\p (v)| . 
\end{equation*}
\end{proof}

If we take $\p(v) =v$ in Theorem \ref{bdd1} and Theorem \ref{compact1}, then we get the results on the multiplication operators $M_\psi$, which are same as Theorem B. 
For the composition operators, by Theorem \ref{bdd1}, Theorem \ref{compact1} for $\psi (v) \equiv 1$, and (ii) of Corollary \ref{cor-compact-Linfty}, and Theorem 4.2 of \cite{ACE}, we get the following result on $C_\p$. 
\begin{corollary} \label{cor-equiv}
Let $\p$ be a self-map of $T$. 
Then the following are equivalent: 
\begin{enumerate}
\item $C_\p : \L \to L^\infty$ is bounded. 
\item $C_\p : \L_o \to L^\infty$ is bounded. 
\item $C_\p : L^\infty \to L^\infty$ is compact. 
\item $C_\p : \L \to L^\infty$ is compact. 
\item $C_\p : \L_o \to L^\infty$ is compact. 
\item $C_\p : \L \to \L$ is compact. 
\item $\p$ has finite range. 
\end{enumerate}
\end{corollary}

%%%%%%%%%%%%%%%%%%%%%%%%%%%%%%%%%%%%%%%%%%%%%%%%
\subsection{Isometries} \

\begin{theorem} \label{iso}
For any function $\psi$ on $T$ and any self-map $\p$ of $T$, $\psi C_\p$ acting from $\L$ (or $\L_o$) to $L^\infty$ is not an isometry. 
\end{theorem}

\begin{proof}
It is enough to prove the statement for the case $\psi C_\p : \L_o \to L^\infty$. 
We assume that $\psi C_\p : \L_o \to L^\infty$ is an isometry. 
If there exists a vertex $w \not\in \p (T)$, then we have that 
\begin{equation*}
1 = \| \psi C_\p \chi_w \|_\infty = \sup_{v \in T} | \psi (v) \chi_w ( \p (v)) | = 0 . 
\end{equation*}
This is a contradiction. 
Hence $\p$ must be surjective and 
\begin{equation*}
\sup \{ | \psi (v)| : v \in \p^{-1} (w) \} = 1 
\end{equation*}
for any $w \in T$. 
Fix a vertex $w \in T$ with $|w| >1$. 
By Theorem \ref{bdd1}, we get 
\begin{eqnarray*}
1 
& = & \| \psi C_\p \| \geq \| \psi (v) |\p (v)| \|_\infty \\ 
& \geq & \sup_{\p (v) =w} | \psi (v) | |\p (v)| \\ 
& = & |w| \cdot \sup_{\p (v) =w} | \psi (v) | = |w| >1 . 
\end{eqnarray*}
This is a contradiction. 
We conclude that $\psi C_\p$ is not an isometry. 
\end{proof}

%%%%%%%%%%%%%%%%%%%%%%%%%%%%%%%%%%%%%%%%%%%%%%%%
\subsection{Boundedness from below and minimum moduli} \ 

We characterize the boundedness from below for the weighted composition operators $\psi C_\p$ acting from $\L$ to $L^\infty$. 
We denote the injectivity modulus $j( \psi C_\p )_{\L \to L^\infty}$ (the surjectivity modulus $k( \psi C_\p )_{\L \to L^\infty}$, resp.) by $j( \psi C_\p )$ ($k( \psi C_\p )$, resp.) in short. 
\begin{theorem} \label{inj-mod-2}
Let $\psi$ be a bounded function on $T$ and $\p$ be a self-map of $T$. 
Suppose that $\psi C_\p$ is bounded from $\L$ to $L^\infty$. 
\begin{enumerate}
\item If $\p$ is not surjective, then $j( \psi C_\p) =0$. 
\item If $\p$ is surjective, then 
\begin{equation} \label{est-inj-2}
\frac{\, 1 \,}{3} \, \inf_{w \in T} \left( \sup_{v \in \p^{-1} (w)} | \psi (v)| \right) 
\leq j( \psi C_\p) 
\leq \inf_{w \in T} \left( \sup_{v \in \p^{-1} (w)} | \psi (v)| \right) . 
\end{equation}
\end{enumerate}
\end{theorem}

\begin{proof} 
(i) can be proved exactly in the same way as in the proof of (i) of Theorem \ref{inj-mod}. 

(ii) Suppose $\p$ is surjective. 
Let 
\begin{equation*} 
M = \inf_{w \in T} \left( \sup_{v \in \p^{-1} (w)} | \psi (v)| \right) . 
\end{equation*}

Since $\| \chi_w \|_{\L} =1$ for any $w$, we get 
\begin{equation*}
j( \psi C_\p) \leq \inf_{w \in T} \, \| \psi C_\p \chi_w \|_\infty = M . 
\end{equation*}

Next we will show $j( \psi C_\p) \geq M/3$. 
Let $f$ be in $L^\infty$ with $\| f \|_\L =1$. 
We put $r = |f(o)|$, then we have that 
\begin{eqnarray}
\| \psi C_\p f \|_\infty 
& = & \sup_{v \in T} | \psi (v) \cdot f ( \p (v))| \nonumber \\ 
& \geq & |f(o)| \sup_{v \in \p^{-1} (o)} | \psi (v) | \ \geq \ r M . \label{ineq-1}
\end{eqnarray}
Since $\sup_{v \in T} |Df (v)| = 1-r$, there exists a sequence $\{ w_n \} \subset T^*$ such that 
\begin{equation*}
|Df (w_n)| = |f(w_n) -f(w_n^-)| > (1-r) \left(1 - \frac{\, 1 \,}{n} \right) . 
\end{equation*}
Then 
\begin{equation*}
\max \{ |f(w_n)| , |f(w_n^-)| \} > \frac{\, 1 \,}{2}(1-r) \left(1 - \frac{\, 1 \,}{n} \right) . 
\end{equation*}

We choose one of the vertex from $\{ w_n , w_n^- \}$ which attains the maximum above, and put $u_n$. 
Since the sequence $\{ u_n \}$ in $T$ satisfies that for any $n$, 
\begin{equation*}
|f(u_n)| > \frac{\, 1 \,}{2}(1-r) \left(1 - \frac{\, 1 \,}{n} \right) , 
\end{equation*}
we have that 
\begin{eqnarray*}
\| \psi C_\p f \|_\infty 
& \geq & |f(u_n)| \sup_{v \in \p^{-1} (u_n)} | \psi (v) | \\ 
& \geq & \frac{\, 1 \,}{2}(1-r) \left(1 - \frac{\, 1 \,}{n} \right) M . 
\end{eqnarray*}
Letting $n \to \infty$, we get 
\begin{equation}
\| \psi C_\p f \|_\infty 
\geq \frac{\, 1 \,}{2}(1-r) M . \label{ineq-2}
\end{equation}
By \eqref{ineq-1} and \eqref{ineq-2}, we have that 
\begin{equation*}
\| \psi C_\p f \|_\infty 
\geq \inf_{0 \leq r \leq 1} \left( \max \left\{ r, \frac{\, 1 \,}{2}(1-r) \right\} \right) M \ = \ \frac{\, 1 \,}{3} M . 
\end{equation*}
Therefore, we obtain $j( \psi C_\p) \geq M/3$. 
\end{proof}

\begin{corollary} \label{cor-bdd-below-Lip-Linfty}
Let $\psi$ be a bounded function on $T$ and $\p$ be a self-map of $T$. 
Suppose that $\psi C_\p$ is bounded from $\L$ to $L^\infty$. 
Then $\psi C_\p$ is bounded below if and only if $\p$ is surjective on $T$ and 
\begin{equation*} 
\inf_{w \in T} \left( \sup_{v \in \p^{-1} (w)} | \psi (v)| \right) > 0 . 
\end{equation*}
\end{corollary} 

We give some necessary conditions and a sufficient condition for $\psi C_\p$ to be surjective. 
\begin{proposition} \label{prop-sur2}
Let $\psi $ be a bounded function on $T$ and $\p$ be a self-map of $T$. 
Suppose that $\psi C_\p$ is bounded from $\L$ to $L^\infty$. 
\begin{enumerate}
\item If there exists $v_0 \in T$ such that $\psi ( v_0 ) =0$, then $k( \psi C_\p) = 0$. 
\item If $\p$ is not injective, then $k( \psi C_\p) =0$. 
\item If $\p$ is injective, then 
\begin{equation*} 
\frac{\, 1 \,}{3} \, \inf_{v \in T} | \psi (v)| 
\leq k( \psi C_\p) 
\leq \inf_{v \in T} | \psi (v)| (1+| \p (v)|) . 
\end{equation*}
\end{enumerate}
\end{proposition}

\begin{proof} 
(i) and (ii) can be proved exactly in the same way as in the proof of Proposition \ref{prop-sur} and (i) of Theorem \ref{sur-mod}. 

(iii) Assume that $\p$ is injective. 
To prove that $\displaystyle k( \psi C_\p) \leq \inf_{v \in T} | \psi (v)|(1+ |\p (v)|)$, we may assume $k( \psi C_\p) > 0$. 
Let $\varepsilon$ be a positive number less than $k( \psi C_\p)$. 
For any $w \in T$, there exists $f \in \L$ such that $\| f \|_\L \leq 1$ and $\psi C_\p f = (k( \psi C_\p) - \varepsilon ) \cdot \chi_w$. 
By the growth condition, we have that 
\begin{eqnarray*} 
k( \psi C_\p) - \varepsilon 
& = & (k( \psi C_\p) - \varepsilon ) \cdot \chi_w (w) \\ 
& = & \psi (w) \cdot f( \p (w)) \\ 
& \leq & | \psi (w)| (1+ | \p (w)|) . 
\end{eqnarray*} 
Since $w$ is arbitrary, we get $ k( \psi C_\p) - \varepsilon \leq \inf_{v \in T} | \psi (v)| (1+ | \p (v)|)$. 
Letting $\varepsilon \to 0$, we get $ k( \psi C_\p) \leq \inf_{v \in T} | \psi (v)| (1+ | \p (v)|)$. 

Next we assume $\inf_{v \in T} | \psi (v)| >0$. 
If $f $ is in $L^\infty$ with $\| f \|_\infty$, then $\| f \|_\L \leq 3$. 
Thus, by Theorem \ref{sur-mod}, we have that 
\begin{eqnarray*}
\psi C_\p \Big( U_\L \Big) 
\supset \psi C_\p \left( \frac{\, 1 \,}{3} \, U_{L^\infty} \right) 
\supset \left( \frac{\, 1 \,}{3} \, \inf_{v \in T} | \psi (v)| \right) \, U_{L^\infty} . 
\end{eqnarray*}
Thus we conclude $\displaystyle k( \psi C_\p) \geq \frac{\, 1 \,}{3} \, \inf_{v \in T} | \psi (v)| $. 
\end{proof}

The upper estimate of $k( \psi C_\p)$ in (iii) of Proposition \ref{prop-sur2} is not best possible. 
We present the example satisfying $\inf_{v \in T} | \psi (v)| (1+| \p (v)|) >0$ but $k( \psi C_\p ) =0$. 
\begin{example}
Let $\p (n) = 2n$, and define a function $\psi$ on $\Z$ by 
\begin{equation*}
\psi (n) = \left\{ \, 
\begin{matrix} 
1 & \ ( n =0 ) \vspace{2mm} \\ 
\displaystyle \frac{\, 1 \,}{n} & \ ( n \not= 0 ) 
\end{matrix} \right. . 
\end{equation*} 
Clearly, $\p$ is injective on $\Z$ and 
\begin{equation*}
\inf_{n \in \Z} | \psi (n)| =0, \ \ \mbox{and} \ \ \inf_{n \in \Z} | \psi (n)| (1+ |\p (n)|) = 2>0 . 
\end{equation*} 
We will show that $\psi C_\p$ is not surjective from $\L$ to $L^\infty$. 

To do this, assume $k( \psi C_\p) > M >0$. 
Let $g(n) = M (-1)^n$. 
Since $g \in M \cdot U_{L^\infty}$, there exists a function $f \in \L$ such that $\| f \|_\L \leq 1$ and $\psi C_\p f =g$. 
For $n \not= 0$, we have that $f(2n) = n g(n) = Mn(-1^n)$. 
Choose a positive integer $n$ satisfying $M (2n+1) >2$. 
Then we have that 
\begin{eqnarray*} 
2 
& \geq & |Df(2n+2)| + |Df(2n+1)| \\ 
& \geq & |f(2n+2) -f(2n)| \\ 
& = & M \cdot | (n+1) (-1)^{n+1} - n(-1)^n)| \\ 
& = & M (2n+1) >2 .
\end{eqnarray*} 
This is a contradiction, therefore $k( \psi C_\p) =0$, that is, $\psi C_\p$ is not surjective from $\L$ to $L^\infty$. 
\end{example}

%\affiliationone{
%\affiliationthree{

\end{document}